\documentclass[reqno]{amsart}
\usepackage{amsthm}
\usepackage{bm}
\usepackage[mathscr]{eucal}
\usepackage{mathrsfs}
\usepackage{mathbbol}
\usepackage{oldgerm,units,color}
\usepackage{hyperref}

\newtheorem{theorem}{Theorem}[section]
\newtheorem{proposition}[theorem]{Proposition}
\newtheorem{definition}[theorem]{Definition}

\newtheorem{corollary}[theorem]{Corollary}

\newtheorem{example}[theorem]{Example}
\newtheorem{remark}[theorem]{Remark}

\newtheorem{note}[theorem]{Note}

\newcommand{\Real}{\mathbb R}

\newcommand{\Net}{\mathbb N}

\newcommand{\one}{\mathbb{1}}
\newcommand{\zero}{\mathbb{0}}

\newcommand{\trop}[1]{\mathcal{#1}}

\newcommand{\tT}{\trop{T}}

\newcommand{\Hom}{Hom}

\newcommand{\al}{\alpha}

    \ifx\proof\undefined
    \newenvironment{proof}{
    \smallskip
    \noindent\emph{Proof.}}{\hfill\(\Box\)
    \bigskip
    } \fi

\newcommand{\ifdef}[3]{\ifthenelse{\equal{#1}{true}}{#2}{#3}}

\def\semiring0{semiring$^\dagger$}
\def\semiringg0{semiring$^{\dagger\dagger}$}
\def\semifield0{semifield$^{\dagger}$}
\def\wV{\bigwedge V_n}
\def\lra{\longrightarrow}

\def\wrV{\bigwedge^rV_n}

\def\zparen{\{z\}}
\def\Cong{\Phi}
\def\bw{\bigwedge}

\def\Gdbl{T(V)_{\operatorname{doub}}\def\ovD{\overline{D}}}

\def\End{{\operatorname{End}}}

\def\semiring0{semiring$^\dagger$}
\def\semirings0{semirings$^\dagger$}
\def\semifield0{semifield$^{\dagger}$}
\def\semidomain0{semidomain$^{\dagger}$}

\def\Dz{D\{z\}}
\def\ovDz{\overline{D}\{z\}}
\def\Dfz{D^f\{z\}}
\def\wM{\bigwedge V}
\def\blamb{{\bm\lambda}}
\def\w{\wedge}

\def\sra{\rightarrow}

\def\ovD{\overline{D}}
\def\wb{[{\mathbf b}]}

\def\bfb{{\mathbf b}}
\def\be{\begin{equation}}
\def\ee{\end{equation}}
\def\Acal{\mathcal A}

\newcommand{\etype}[1]{\renewcommand{\labelenumi}{(#1{enumi})}}
\def\eroman{\etype{\roman}}

\def\Null{\operatorname{Null}}

\newtheorem*{nothma}{\textbf{Theorem A}}
\newtheorem*{nothmb}{\textbf{Theorem B}}
\newtheorem*{nothmc}{\textbf{Theorem C}}
\newtheorem*{nothmd}{\textbf{Theorem D}}
\newtheorem*{nothme}{\textbf{Theorem E}}

\def\tTz{\mathcal T_\zero}

\def\({\left(}
\def\){\right)}

\def\a{\alpha}

\def\one{\mathbb{1}}
\def\zero{\mathbb{0}}

\def\ctw{\cdot_{\operatorname{tw}}}

\def\rb{b}
\def\bfb{\textbf{\rb}}

\def\supp{\operatorname{supp}}

\def\Q{\mathbb Q}

\def\Z{\mathbb Z}

\newtheorem{thm}[theorem]{Theorem}

\newtheorem*{thm*}{Theorem}
\newtheorem*{dig*}{Digression}
\newtheorem{lem}[theorem]{Lemma}
\newtheorem{rem}[theorem]{Remark}
\newtheorem{prop*}{Proposition}

\newtheorem{prop}[theorem]{Proposition}
\newtheorem{defn}[theorem]{Definition}
\newtheorem{cexample}[theorem]{Counterexample}
\newtheorem*{examp*}{Example}
\newtheorem*{examples*}{Examples}
\newtheorem*{remark*}{Remark}
\newtheorem*{defn*}{Definition}
\newtheorem*{note*}{Note}

\begin{document}
\title[Grassman semialgebras and the Cayley-Hamilton theorem] {Grassman semialgebras and the Cayley-Hamilton theorem}

\author{Letterio Gatto}
\address{Dipartimento di Scienze Matematiche, Politecnico di Torino, C. so Duca degli Abruzzi 24, 10129 Torino - Italia }
\email{letterio.gatto@polito.it}
\author{Louis Rowen}
\address{Department of Mathematics, Bar-Ilan University, Ramat-Gan 52900,
Israel} \email{rowen@math.biu.ac.il}

\thanks{The first author was partially supported by INDAM-GNSAGA and by PRIN "Geometria sulle variet\`{a} algebriche" \, Progetto di Eccellenza\, Dipartimento\, di\, Scienze\, Matematiche, 2018--2022 no.
E11G18000350001. The second author is supported in part by the
Israel Science Foundation, grant  No. 1207/12 and his visit to
Torino was supported by the ``Finanziamento Diffuso della Ricerca'',
grant no. 53$\_$RBA17GATLET, of Politecnico di Torino}

\subjclass[2010]{Primary  15A75,   16Y60,  15A18; Secondary 12K10,
14T10 }



\keywords{Cayley-Hamilton theorem,   exterior semialgebras,
Grassmann semialgebras, Hasse-Schmidt derivations, differentials,
eigenvalues, eigenvectors Laurent series, Newton's formulas,
 power series, semifields,  systems,
semialgebras, Tropical algebra, triples.}


\begin{abstract}
We develop a theory of Grassmann semialgebra triples using
Hasse-Schmidt derivations, which formally generalizes results such
as the Cayley-Hamilton theorem in linear algebra, thereby providing
a unified approach to classical linear algebra and tropical algebra.
\end{abstract}

\maketitle


\section{Introduction}
\numberwithin{equation}{section}

The main goal of this paper is to define and explore the semiring
version of the theory \cite{Ga2,GaSc,GaSc2} of the first author
concerning the Grassmann {exterior} algebra, viewed more generally
in terms of negation maps and systems, continuing the approach of
\cite{Row16}. This provides a robust structure which unifies and
generalizes four major versions of Grassmann structures, cf.~ Note
~\ref{unify}, and provides a generalization of  the Cayley-Hamilton
theorem in Theorem~\ref{CHthm}. In the process we investigate
Hasse-Schmidt derivations on Grassmann systems. The version given in
Theorem~\ref{ge22} (over a free module $V$ over an arbitrary
semifield) is the construction which seems to ``work.''

Generalizing negation in Definition~\ref{negmap} to the notion of a
``negation map'' $(-)$ satisfying all the usual properties of
negation except $a(-)a = 0$ (studied in \cite{Row16}), we find that
the Grassmann semialgebra of a free module $V$, described in
Theorem~\ref{ge22}, has a natural negation map  on all homogeneous
vectors,  with the ironic exception of $V$ itself, obtained by
switching two tensor components.

\begin{nothma}[Theorem~\ref{ge22}]
 $ \bigoplus _{n\ge 2} V^{\otimes (n)}$ has a
negation map $(-)$ satisfying $$ b_{\pi(i_1)}\cdots \bar
b_{\pi(i_t)}=(-)^{\pi}  b_{i_1}\cdots \bar  b_{i_t},$$ for
$b_{i_j}\in V.$ 
\end{nothma}

This provides us ``enough'' negation, coupled with the relation
$\preceq_\circ$ of Definition~\ref{precmain0} (designed to replace
equality), to adapt the methods of \cite{GaSc2} to negation maps and
systems of \cite{Row16} to study $T(V)$, with $\tT$   the nonzero
simple
 tensors.

 One obtains a Grassmann algebra by modding out all elements of the form $v \otimes v$, $v\in
 V$. A weaker version   is obtained when
  we take a given base $\{b_0, b_1,
\dots, b_{n-1}\}$ of $V$ and mod only   the $b_i \otimes b_i, \ 0\le
i <n$. One can enhance these by means of a ``symmetrizing''
construction which provides a negation map and a system for any
Grassmann algebra:

\begin{nothmb} [Theorems~\ref{gr1}, \ref{gr2}, \ref{gr3}] There is a symmetrized triple
$(\mathfrak G(V)^\natural,\widehat {\tT},(-)_{\operatorname{sw}})$,
 together with an embedding of triples    $ (\mathfrak G(V)_{\ge
2},\tT_{\mathfrak G(V)_{\ge 2}},(-)) \to (\mathfrak
G(V)^\natural,\widehat {\tT},(-)_{\operatorname{sw}})$  given by  $c
\mapsto (c,0)$. An analogous assertion holds for $\mathfrak
G(V)^\diamondsuit.$
\end{nothmb}

\begin{note}\label{unify}
There are four major applications of this Grassmann triple:

  \begin{enumerate}\eroman
\item The classical situation is recovered, where $V$ is a vector space and $(-)$
is the classical negation.
\smallskip
\item  The tropical situation is treated in \cite{GG} to study matroids, taken over the max-plus algebra, cf.~Remark~\ref{zerosum}
where $(-)$ is the identity map.
\smallskip
\item  The supertropical situation, as described in   Theorem~\ref{gr1},  provides an alternate language
for the theory of \cite{GG}.
\smallskip
\item  The symmetrized situation, as described   in \S\ref{symm},  provides a nontrivial negation map    for any Grassmann semialgebra.
\end{enumerate}
\end{note}

 The rest of \S\ref{grT1} is dedicated to laying the groundwork
for further research in Grassmann semialgebras, and the reader only
interested in the applications of this paper could go on to
\S\ref{basics}, where we study properties of derivations. There we
learn how to associate to an endomorphism $f$ of  an $\Acal$ free
module $V_n$  a Hasse-Schmidt derivation $\Dz:\bw V_n\sra \bw V_n
[[z]]$,  on the  Grassmann semi-algebra, i.e $\Dz (u\w v)=\Dz u\w
\Dz V$  and $\Dz_{|_{V_n}}=\sum_{j\geq 0}f_jz^j$. We  suitably
construct, starting from the data of $\Dz$, an operator valued
polynomial $\ovDz:\bw V_n\sra \End(\bw V_n)[z]$ showing in detail:

\begin{nothmc}[Theorem~\ref{prech1}] The  polynomial $\ovDz$ is a quasi-inverse of $\Dz$, in the sense that $ \ovDz  \Dz u\succeq u  $ for all $u\in \bw V_n$. \end{nothmc}

 Our main result,
Theorem~\ref{prech},
 describes the relation between a Hasse-Schmidt
derivation $\Dz$ and its ``quasi-inverse'' $\overline {D} \{z\}$
defined in such a way to yield:

\begin{nothmd}[Theorem~\ref{prech}]
 $\overline {D} \{z\} (\Dz  u\w v)\succeq u\w \overline {D} \{z\}
v.$
 \end{nothmd}

In the classical case, one gets equality as shown in
Remark~\ref{eq:CHform17}, so we recover \cite{GaSc2}.  
 Our main application in this paper is a
generalization of the Cayley-Hamilton theorem to semi-algebras.

\begin{nothme}[Theorem~\ref{CHthm}]
 \be
\left((D_nu+e_1D_{n-1}u+\cdots+e_{n}u)\w v\right)
(-)\left((D_nu+e'_1D_{n-1}u+\cdots+e'_{n}u)\w v \right) \succeq 0
 \ee for all $u\in \bw^{>0}V_n$.
 \end{nothme}

We obtain true equality in these theorems  by modding out all
elements of the form $v \otimes v$. Again we recover \cite{GaSc2} in
the classical
 case.

\subsection{Basic notions}\label{bn}$ $

Much of this section is a review of \cite{Row16}, as summarized in
\cite{Row17}, and also as in \cite{JMR}.   As customary,
$\Net$~denotes the natural numbers including 0, $\Net^+$ denotes
$\Net \setminus \{0\},$ $\Q$ denotes the rational numbers, and
$\Real$ denotes the real numbers, all ordered monoids under
addition.

A \textbf{\semiringg0}  $(\mathcal A, +, \cdot)$ is an additive
abelian semigroup $(\mathcal A, +)$ and multiplicative semigroup
$(\mathcal A, \cdot)$ satisfying the usual distributive laws. A
\textbf{\semiring0}  $(\mathcal A, +, \cdot, \one)$ is a \semiringg0
with a multiplicative unit~$\one$. (Thus, an ideal of a \semiring0
is a \semiringg0.) A  \textbf{\semifield0}  is a \semiring0 whose
multiplicative  monoid is  a group. A \textbf{semiring} (resp.~
\textbf{semifield}) \cite{golan92} is a \semiring0 (resp.~
 \semifield0) with an absorbing element $\zero$  formally adjoined.

\begin{defn}\label{modu2}   A $\tT$-\textbf{module} over a set $ \tT$
 is an additive monoid $( \mathcal A,+,\zero)$ with a scalar
multiplication $\tT\times \mathcal A \to \mathcal A$ satisfying the
following axioms, $\forall k \in \Net,$ $a \in \tT,$ $b,b_j \in
\mathcal A$:

 \begin{enumerate}\label{distr32}\eroman
   \item (Distributivity over $\tT$): $a (\sum _{j=1}^k b_j) =     \sum _{j=1}^k (a b_j). $
\item  $a\zero_{\mathcal A }= \zero_{\mathcal A }$.
\end{enumerate}


A $\tT$-\textbf{monoid module}   over a multiplicative monoid $\tT$
is a $\tT$-module satisfying the extra conditions $$\one_\tT b = b,
\qquad (a_1a_2) b = a_1(a_2b), \quad \forall a_i \in \tT, \ b \in
\mathcal A.$$
\end{defn}

A $\tT$-\textbf{semiring} is a semiring that is also a $\tT$-monoid
module over a given multiplicative submonoid~$\tT$. This paper only
concerns   $\tT$-semirings, which are closely related to blueprints
in \cite{Lor1}.  We put $\tTz = \tT \cup \{ \zero\}.$

 Tensor products over
semirings \cite{Ka1,Ka2,Tak} are analogous to tensor products over
rings. The tensor product $\mathcal M_1 \otimes _{\mathcal A}
\mathcal M_2$ of right and left $\mathcal A$-modules $\mathcal M_1$
and $\mathcal M_2$ is $(\mathcal F_1 \oplus \mathcal F_2)/\Cong,$
where $\mathcal F_i$ is the free module (respectively right or left)
with base~$\mathcal M_i$, and $\Cong$ is the congruence generated by
all
 \begin{equation}\label{defcong}\bigg(\big(\sum_j x_{1,j}, \sum_k x_{2,k}\big), \sum_{j,k}\big( x_{1,j},  x_{2,k}\big)\bigg) ,\quad \bigg((  x_1 a, x_2 ), (x_1,a
 x_2
)\bigg)\quad  \forall x_i, x_{i,j},x_{i,k}\in \mathcal M_i,  \ a \in
\mathcal A.
\end{equation}

\subsection{Negation maps, triples, and systems}

\begin{defn}\label{negmap}
 A \textbf{negation map} on  a $\tT$-module $M$ over a
 given set $\tT$
is a
 semigroup isomorphism
$(-) :M \to M$ of order~$\le 2,$  written $b\mapsto (-)b$, which
also
 respects the $\tT$-action in the sense that
$$(-)(ab) = a((-)b)$$ for $a \in \tT,$ $b \in M.$

 A \textbf{\semiringg0 negation map} on  a \semiringg0 $\mathcal A$
is a negation map which   satisfies  $(-)(ab) = a((-)b) = ((-)a)b $
for all $a ,b \in \mathcal A.$
 \end{defn}
%
In the classical case the negative is a negation map.  For tropical
algebra, one could just take $(-)$ to be the
identity map, but we want a less trivial example. 

We write $ a (-)b $ for $a+ ((-)b)$, $(\pm )a$ for $\{ a, (-)a\},$
and $ a ^\circ$ for $ a (-)a$, called a \textbf{quasi-zero}. The
set~$M ^\circ$ of quasi-zeroes is an important $\tT$-submodule of $M
$. When $\mathcal A$ is a semiring, $\mathcal
 A^\circ$ is  an ideal.

We define $(-)^\zero a$  to be $a$ and, for $k \in \Net$, we
inductively define $(-)^k a$ to be $(-)((-)^{k-1}a).$

\begin{lem}\label{repneg}  $((-)^k a )((-)^{k'} a') = (-)^{k+k'} (aa')$  for $a,a' \in \mathcal
A.$
 \end{lem}
\begin{proof}
$((-)^k a )((-)^{k'} a') = (-)((-)^{k-1}a)((-)^{k'} a')  =
(-)((-)^{k+k'-1}(aa')) = (-)^{k+k'} (aa'),  $ by induction on $k$.
\end{proof}


\begin{defn}\label{modu21}
A \textbf{pseudo-triple} is a collection $(\mathcal A, \tT, (-)),$
where  $(-)$ is a negation map on both $\tT$ and $\mathcal A $, and
$\mathcal A$ is a $(\tT,(-))$-module.

 A $\tT_{\mathcal A}$-\textbf{pseudo-triple} $(\mathcal A, \tT_{\mathcal
A}, (-))$ is a $\tT$-module $\mathcal A$, with $\tT_{\mathcal A}$
designated as a distinguished subset, and a negation map $(-)$
satisfying $(-)\tT_{\mathcal A} = \tT_{\mathcal A}.$
 A
$\tT_{\mathcal A}$-\textbf{triple}, called a \textbf{triple} when
$\tT$ is understood, is a $\tT_{\mathcal A}$-pseudo-triple, in which
$\tT_{\mathcal A} \cap \mathcal A^ \circ = \emptyset$ and
$\tT_{\mathcal A}$ generates $(\mathcal A \setminus \{\zero\},+).$

Since $\tT_{\mathcal A}$ is usually clear from the context, we abuse
notation and write $\tT$ for $\tT_{\mathcal A}\subseteq \mathcal A$.
A triple is \textbf{uniquely negated} when for any $a \in \tT$, $a+b
\in \mathcal A^\circ$ implies  $b = (-)a.$
\end{defn}

%
%
%

The structure is rounded out with the following relation.

\begin{defn}\label{precmain0} Define the $\circ $-\textbf{surpassing relation} $\preceq_\circ $ on a
module $M$ with negation map by $a_0 \preceq_\circ a_1$ if $a_1 =
a_0 + d$ for some $d\in M ^\circ$.\end{defn}

A uniquely negated triple $(\mathcal A, \tT, (-))$ together with a
surpassing relation $\preceq$ is called a
\textbf{system}\footnote{In  \cite{JuR1}, in the systemic setting, a
more general notion of surpassing map $\preceq $ is used, and
$\mathcal A _{\Null}$
 is introduced  which equals  $\mathcal A ^\circ$ when $\preceq =\preceq
 _\circ;$ here, using Lemma~\ref{biG2} as justification, we use  $\mathcal A ^\circ$ and $\preceq =\preceq
 _\circ$ to avoid complications.}.

\begin{rem}\label{trans}The relation $\preceq_\circ$ on a system restricts to   equality on $\tT$, by
\cite[Proposition~4.4]{Row16}. In fact
 $\preceq_\circ$ is used to replace equality when we work with triples,
 and identities in classical algebra can often be replaced by
 relations expressed in terms of  $\preceq_\circ$,
 by means of the transfer principle of \cite{AGG2}, formulated for systems in~\cite[Theorem~6.17]{Row16}.
\end{rem}

\subsection{Functions to $\mathcal A$}$ $

The next construction, discussed in \cite[\S 4.2]{JuR1}, enables us
to describe power series in a structural context. From now on, we
suppose $(S,+)$ is a semigroup, often   $(\mathbb N,+,0)$. Given a
triple $(\mathcal A,\tT,(-))$, $\mathcal A^S$~denotes the maps from
$S$ to $\mathcal A$, and $\mathcal \tT^S$ denotes the nonzero maps
of $\mathcal A^S$ sending $S$ to $\tT$. For example, for $c \in
\mathcal A,$ the \textbf{constant function} $\tilde c $ is given by
$\tilde c (s) = c$ for all $s \in S.$

We modify the definition of support from
\cite[Definition~4.2]{JuR1}.

\begin{defn}
Given $f\in  \mathcal A ^S$ we  define its \textbf{support} $\supp
(f):=\{ s \in S: f(s ) \ne \zero \},$ 
and $\supp ( \mathcal A ^S)$ for $\cup \{\supp (f): f \in
 \mathcal A ^S\}$.
 \end{defn}
%

 \begin{defn}\label{convn} A set ${\mathcal A}\flat $ of maps $f: S \to \mathcal A $ is
\textbf{convolution admissible} if for each $f,g \in {\mathcal
A}\flat$ and $s \in S$ there are only finitely many $s' \in \supp
(f),$ $s'' \in \supp (g),$ with $s'+s'' = s.$
\end{defn}

\begin{example}\label{power1}
${\mathcal A}\flat$ is convolution admissible whenever $S =
\Net^{(I)}$ (the direct sum) for some index set $I$, since the
condition of Definition~\ref{convn} already holds
 in $S$.
 \end{example}

\begin{defn}\label{grsem} Suppose ${\mathcal A}\flat$  is a convolution admissible set.
 The \textbf{convolution product}
${\mathcal A}\flat \times {\mathcal A}\flat\to {\mathcal A}\flat$ is
given by defining $fg$ to be the function satisfying
$$fg(s) = \sum _{s'+s'' = s} f(s')g(s'').$$
%

%
%
\end{defn}

The intuitive way to receive a negation map on ${\mathcal A}\flat$
from a negation map on $\mathcal A$  would be to define $((-)f)(s) =
(-)(f(s));$ these maps also are convolution admissible,
so  one would expand ${\mathcal A}\flat$ to include them.

\
\subsection{Graded semirings and modules}$ $

We want to grade semirings and their modules. We define direct sums
in the usual way.

\begin{defn}

An $L$-\textbf{graded} $\tT$-\semiring0  is a $\tT$-\semiring0
$\mathcal R $ which also is an $L$-graded $\tT$-module $\mathcal R :
= \oplus _{\ell \in L} \mathcal R _\ell$ for semigroups $(\mathcal R
_\ell,+)$ satisfying the following conditions, where $\tT_ \ell =
 \tT\cap  \mathcal R
_\ell:$
\begin{enumerate}\eroman
    \item $\tT = \cup _{\ell \in L} \tT _\ell;$
\item $\mathcal R_\ell \mathcal R
_{\ell'} \subseteq \mathcal R_{\ell +\ell'},$ $\forall \ell, \ell'
\in L.$
 \end{enumerate}

\end{defn}

Note that $\mathcal R_\zero $ is a $\tT_\zero$-module, and also a
\semiring0, over which each $\mathcal R_\ell \cup \{\zero\}$ is a
module.

When we turn to Grassmann  semialgebras, $L$ will be ordered with a
minimal element~$0$; one could take $L = \mathbb N$, for example. We
write $\mathcal M^{> 0}: = \oplus _{\ell
>0} \mathcal M _\ell$, a submodule of $\mathcal M$ lacking the constant
component. Then $\mathcal R^{> 0} $ is a sub-\semiringg0 of
$\mathcal R$.

%
%
%

\subsubsection{Super-semialgebras}$ $

Here is an interesting special case.

\begin{defn} \label{Liesup}  A  \textbf{super-semialgebra}  is a $\mathbb
Z_2$-graded semialgebra $\mathcal A: = \mathcal A_0 \oplus \mathcal
A_1$, i.e., satisfying \textbf{twist multiplication}:

\begin{equation}\label{twi7}(a_0,a_1)(a'_0,a'_1) = (a_0a'_0 + a_1 a'_1,
a_0a'_1 + a_1 a'_0). \end{equation}
\end{defn}

A natural way of getting a $\Z_2$-grade from an $\Net$-graded
semialgebra is to take the 0-grade to be the set of even indices and
the 1-grade to be the set of odd indices.

 \subsubsection{The  power series  semiring of a graded \semiring0}$ $

From now on, we take $S = \Net$ as in Example~\ref{power1}, which is
equivalent to the following.

\begin{defn}\label{psr} Given an $\Net$-graded $\tT$-semiring $\mathcal R $ with respect to the semigroup
$(\Net,+)$, we define the power series semiring $\mathcal R[[z]]$
over a central indeterminate $z$, in the usual way as possibly
infinite formal sums (with convolution product), and its
sub-semiring $\mathcal R[z] = \sum _j \mathcal R _j z^j.$
\end{defn}

\begin{lem}\label{psr13} $\mathcal R[[z]]$ and  $\mathcal R[z]$  are indeed semirings.
Both
 $\mathcal R[[z]]$ and  $\mathcal R[z]$ are graded by the powers of
~$z$. \end{lem}
\begin{proof} $\mathcal R[[z]]$ satisfies the axioms of a semiring,
by the customary verification, and its subset $\mathcal R[z] $ is
closed under addition and multiplication, so both are semirings. The
last assertion follows from the fact that $\mathcal R _j   z^j
\mathcal R _k   z^k \subseteq \mathcal R _{j+k}   z^{j+k}.$
\end{proof}

\begin{defn} \label{hsderiva} Suppose $ \mathcal A$ is a semialgebra over a commutative  base
\semiring0. We write $\End ( \mathcal A)$ for the set of module maps
$ \mathcal A \to \mathcal A.$ Given $D \in \End ( \mathcal A) ^S,$
we write $D_s$ for the map given by $s \mapsto D(s)$
  and $s' \mapsto \zero, \ \forall s'\ne s$. In the other direction, given $f_s:  \in \End (\mathcal A)$, $s
\in S$, each of singleton support $\{ s\}$ with the $\{ s\}$
distinct, define $D^f \in \End ( \mathcal A) ^S$ via $D^f (s) =
f_s.$

\begin{rem}\label{ge3} $D^f(s)(ab) = D^f(s)(a)D^f(s)(b), \forall a,b\in \mathcal A
,$  under the usual product, seen by matching terms  in the left
side and the right side.
\end{rem}

\subsection{Higher  derivations}$ $

This discussion applies to any semialgebra $\mathcal A$, not
necessarily associative and not necessarily a triple. A
\textbf{derivation} $\delta: \mathcal A \to \mathcal A$ is a map in
$\End ( \mathcal A)$ satisfying $\delta(ab) = a \delta(b) +
\delta(a)b.$ The following concepts were introduced by Hasse and
Schmidt~\cite{HaS} and studied further by
Heerema~\cite{He1,He2,He3}.


 For  $S $ convolution admissible
       (not necessarily associative), a homogeneous map $D$ in
$(\End \mathcal A)^S$  is called  a
 \textbf{higher  derivation}
  of
  $A$ if it satisfies the conditions:
  \begin{enumerate}
  \eroman

\item
$ D_s   (ab)= \sum_{s'+s''=s}D_{s'}(a) D_{s''}(b),\quad \forall s\in
S, \quad \forall a,b\in \mathcal A .$
\item $D_0
 = \one$  (the identity map on
  $A$).
  \end{enumerate}
  \end{defn}

Property   (i)  is  called  the
  \textbf{Leibniz   rule}, obtained from the
   more  familiar  Leibniz
rule for derivations for $S = \Net$ by dividing by $k !$. 
%

  \cite[pp.~190-191]{HaS} indicates how to define a higher derivation
  $D$. We have a somewhat different take, along classical lines.
We consider semialgebras over  semifields containing $\Q_{>0}$, for
the following definition to make sense. Given a map $f: \mathcal A
\to \mathcal A[[z]]$ we define its \textbf{exponential} $\exp(f) =
\sum _{j\ge 1} \frac {f^j}{j!} : \mathcal A \to \mathcal A[[z]]$. It
is well-known that the exponential of a derivation is a
homomorphism. In fact we have:

\begin{lem}\label{ge4} If $d_1, d_2, \dots $ is a sequence of
derivations, then $\sum d_k z^k : \mathcal A \to \mathcal A[[z]] $
satisfies Leibniz' rule. Its exponential is a semialgebra
homomorphism: $D := \sum D_ r z^r := exp( \sum _{k\ge 1} d_k z^k) :
\mathcal A \to \mathcal A[[z]]. $
\end{lem}
\begin{proof} Given in \cite{W} and \cite[Propositions 3.4.2 and 3.4.3]{GaSa}.
The proof evidently extends to semialgebras over  semifields
containing $\Q_{>0}$.
\end{proof}

Matching coefficients in Lemma~\ref{ge4}, one gets precisely the
Schur polynomials associated to the sequence  $d_1, d_2, \dots $. In
particular: $$ D_1 = d_1, \quad D_2 = \frac{d_1^2} 2 + d_2, \quad
D_3 = \frac{d_1^3}{ 3!} + d_1d_2 + d_3, \quad D_4 = \frac{d_1^4}
{4!}+ \frac 1 2   d_1^2 d_2 + \frac 1 2  d_2^2 + d_1d_3 + d_4,$$
defines a higher derivation $D$.


When each $d_s = \delta$ for a given derivation $\delta$, we call
$D$ the \textbf{higher derivation} of $ \delta$.

\section{Grassmann semialgebras}\label{grT1}

  Suppose $\mathcal A$ is a
commutative semiring and ~$V$ is an $\mathcal A$-module. $\mathfrak
H$ will denote an $\mathcal A$-semialgebra generated by $V$. (Often
$\mathfrak H$ will be the tensor algebra $T(V)$ defined below.) We
write ${\tT_{\mathfrak H}}_k$ for the products of length~$k$ of
elements of $V$,    and $\mathfrak H_{\ge k}$ for the ideal $\sum
_{j\ge k} {\tT_{\mathfrak H}}_{j}$ (with repetitions). Thus
$\mathfrak H = \mathcal A + V + \mathfrak H_{\ge 2}$. The elements
of ${\tT_{\mathfrak H}}_k$ and ${\tT_{\mathfrak H}}_l$ will satisfy
$w _k w' _l = (-1)^{kl}w' _l w_k ,$ leading to the subject of our
study.

\begin{defn}\label{Grass1} A \textbf{Grassmann}, or \textbf{exterior},
semialgebra, over a \semiring0 $\mathcal  A$  and an $\mathcal
A$-module $V$,
 is a semialgebra $\mathfrak H$ generated by $\mathcal A$ and
 $V$, as above,
together with  a negation map on  $\mathfrak H_{\ge 2}$ and an
associative \textbf{wedge product} $\w: \mathfrak H \times \mathfrak
H \to \mathfrak H$ satisfying
     \begin{equation}\label{G2} v_1\w v_2 = (-)
 v_2 \w v_1 \qquad  \text{for} \qquad v_i \in V.
 \end{equation}

\end{defn}

Thus $v_{\pi(1)}\w \cdots \w   v_{\pi(t)} = (-)^{\pi} v_1\w \cdots
\w v_t$ for $t \ge 2,$ where  $(-)^{\pi}$ denotes the sign  of the
permutation. This ties in with the theory of triples since, taking
$\tT(V)$ to be the nonzero products of elements of $V$, then for
$k\ge 2,$ $(\mathfrak H_{\ge 2}, \tT(V)_{\ge 2},(-))$ often is a
triple, for $(-)$ suitably defined (as in Theorem~\ref{ge22}). Since
$\tT(V)_{\ge 2} \subset \mathfrak H_{\ge 2}$, it can be bypassed by
restricting functions, but the negation map $(-)$ will play a
crucial role.

\begin{lem}\label{welld} If $V$ is spanned by $\{ b_i : i\in I\}$, then to verify  the   Grassmann
relation \eqref{G2} it is enough to check   that $$  b_i \wedge b_j
= (-) b_j \wedge b_i, \qquad \forall i,j \in I.$$
\end{lem}
\begin{proof} Distributivity yields
 $$\left(\sum \a_i b_i\right)\wedge \left(\sum \beta_j b_j\right) = \sum \a _i \beta_j b_i \wedge b_j = (-)\sum \a _i \beta_j b_j \wedge b_i =
(-) \left(\sum \beta_j b_j\right)\wedge\left(\sum \a_i b_i\right),$$
 yielding the assertion.\end{proof}

We  write $v^k$ for $v \wedge \dots \wedge v$ taken $k$ times.

\begin{lem} $(\sum \a_i a_i)^2 \succeq_\circ \sum \a _i^2 a_i^2$ for
 any Grassmann semialgebra.
\end{lem}
\begin{proof} $(\sum \a_i a_i)^2 = \sum \a _i^2 a_i^2 + \sum _{i<j} \a_i \a_j (a_i \wedge a_j + a_j
\wedge a_i) $.
\end{proof}

To obtain Grassmann semialgebras via Definition~\ref{stdG} below, we
follow the familiar construction of the Grassmann  algebra over a
module $V$, but with modifications necessitated by working over
semirings.

Accordingly, as in \cite[Remark~6.35]{Row16} and
{\cite[Definition~6.10]{JuR1}}, we define the \textbf{tensor
semialgebra} $T(V) = \bigoplus _n V^{\otimes (n)}$ (adjoining a copy
of $\mathcal A$ if we want to have a unit element), with the usual
multiplication $v v' := v\otimes v'$.

\begin{thm}\label{ge22}
Write $T(V)_{\ge 2}$ for $ \bigoplus _{n\ge 2} V^{\otimes (n)}$.
 Then  $T(V)_{\ge 2}$ has a
negation map $(-)$ satisfying $$b_{\pi(i_1)}\otimes \cdots \otimes
 b_{\pi(i_t)} = (-)^{\pi} b_{i_1}\otimes \cdots
\otimes b_{i_t},$$ for $b_{i_j}\in V.$ 
\end{thm}
\begin{proof} By Lemma~\ref{welld}, we may take a generating set
$\{b_i : i \in I\}$ of $V$, where $I$ is an ordered index set. 
 We define a negation on~$V \otimes V$ by
$(-)b_i \otimes b_j = b_j \otimes b_i.$ (This is possible since it
preserves the bilinear relations  defining the tensor product.)
Since this is homogeneous of degree 2, it defines a negation on
$\mathfrak G(V)_{2}$ given by $(-)b_i \otimes b_j = b_j  \otimes
b_i.$ When $i<j$ we thus rename $b_j \otimes  b _i$ as $(-)b_i
\otimes  b_j$. It is easy to see that this is the same as defining a
reduction procedure. Thus $b_{\pi(i_1)}\cdots b_{\pi(i_t)} \mapsto
(-)^{\pi} b_{i_1}\cdots b_{i_t},$ where $\pi$ is the permutation
rearranging the indices $i_1\dots ,i_t$ in ascending order. We get
$(-)^{\pi}$
  by writing $\pi$ as a product of transpositions;
since $(-)^{\pi}$ is independent of the way we write $\pi$ in this
manner, our reduction procedure is well-defined, cf.~\cite{New}.

\end{proof}

%

We continue to develop the Grassmann theory. We can eliminate many
occurrences of $(-)$ in our formulas by switching two of the $ b_i.$
The tricky part is dealing with degree 1, i.e., in $V$ itself, where
we cannot perform this switch. But issues like determinants and
linear independence of $n$ vectors are trivial for $n=1$, thereby
enabling us to forego $(-)$ on elements of degree 1. In this manner,
our way out in \S\ref{basics} is to focus on elements of degree
$>1$.


\begin{defn}\label{biG}
  $\tT_{\operatorname{even}}^{\ge 2}$ is the set of
all even products of elements of $V$, not including the constants
$\mathcal A$, $\mathfrak G_{\operatorname{even}}^{\ge 2}$ is the
submodule of~$\mathfrak G$ generated by
$\tT_{\operatorname{even}}^{\ge 2}$, $\tT_{\operatorname{odd}}$ is
the  set of all odd products of elements of $V$, and $\mathfrak
G_{\operatorname{odd}}$ is the submodule of $\mathfrak G$ generated
by $\tT_{\operatorname{odd}}$.
\end{defn}

\begin{lem} If $v\in \mathfrak G_i$ and  $v'\in \mathfrak G_j$ for $i,j \ge 1$ then
 \begin{equation}\label{G27} v \wedge v' = (-)^{i+j}
v' \wedge v,
 \end{equation}
 where $(-)$ is given as in Theorem~\ref{ge22}.
\end{lem}
\begin{proof} Easy induction on $i$ and $j$.
\end{proof}

\begin{defn}
 ${T(V)^\circ}^\sharp $ is the ideal of $T(V)$ generated by $T(V)^\circ $ and all elements $v\otimes
v, \, v \in V.$\end{defn} (This is just $T(V)^\circ $ when $\frac 12
\in \mathcal A$ since then $v\otimes v = (\frac 12 v\otimes
v)^\circ$). 
We now weaken ~$\preceq_\circ$.
\begin{defn} (Supplanting Definition~\ref{precmain0})
 $a_0 \preceq a_1$ in $T(V)$ if $a_1
= a_0 + d$ for some $d\in {T(V)^\circ}^\sharp$.
\end{defn}

\subsubsection{The \textbf{standard   Grassmann semialgebra}} $ $

 Recall that the way to define factor structures in universal
algebra (in particular, for \semirings0 or modules over \semirings0)
is to mod out by a congruence. 

%
%

 \begin{thm}\label{impl}
If $v^2 = \zero $ for all $v$ in $V$, then for any permutation $\pi$
and all $v_i \in V,$  \begin{enumerate}\eroman
\item $v \w v_1 \w \dots \w v_n\w v=\zero.$
\item $v_{\pi(1)}\w
\dots \w v_{\pi(n)} = v_1 \w \dots \w v_n$ if $\pi$ is even;
\item  $v_1 \w \dots \w v_n + v_{\pi(1)}\w
\dots \w v_{\pi(n)} = \zero$ if $\pi$ is odd.
 \end{enumerate}
 Thus the only quasi-zeros are 0.
\end{thm}
\begin{proof}  Linearizing yields
$$\zero = (v_1+v_2)^2= v_1^2 + v_2^2 + v_1\w v_2 + v_2 \w v_1 = \zero +
\zero + v_1\w v_2 + v_2\w v_1,$$  so  $v_1\w v_2 + v_2\w v_1=
\zero.$ Now (i) is by induction on $n$, since
 $v \w v_1 \w \dots \w v_n\w v =   v \w v_1 \w \dots \w v_n\w +\zero = (v \w v_1 + v_1 \w v)\w \dots \w v_n\w v
 =\zero.$

To get (ii) and (iii) we write $\pi $ as a product of transpositions
$\pi_1 \cdots \pi _k$ of the form $(i,i+1)$. If $k=2$ then $v_1 \w
v_2 \w v_3 = v_1 \w v_2 \w v_3 + (v_2\w (v_1\w v_3 +  v_3 \w v_1)=
(v_1 \w v_2 + v_2\w v_1)\w v_3 + v_2\w v_3 \w v_1= v_2\w v_3 \w
v_1,$ and then we have (ii) for all even $k$.

For $k$ odd, we use (ii) to reduce to $ v_1 \w v_2 \w \dots \w v_n\,
+\, v_2 \w v_1 \w \dots \w v_n = \zero$.

The last assertion follows by using these equalities to reduce every
quasi-zero until reaching $\zero.$
\end{proof}

 \begin{defn}\label{stdG}
  The \textbf{standard   Grassmann
semialgebra} $\bigwedge V$ with respect to a given generating set
$\{b_i: i \in I\}$ of $V$, also denoted $\mathfrak G(V)$, is   $
T(V)/\Phi $, where $\Phi$ is the congruence generated by  $(v \w
v,\zero),$ $\forall v \in V.$ Accordingly $\mathfrak G(V)_{k}$ is
$T(V)_{k}/\Phi$, and $ \mathfrak G(V)_{\ge 2}= T(V)_{\ge 2}/\Phi$.

The \textbf{standard   Grassmann triple} is $(\mathfrak G(V)_{\ge
2},\tT_{\mathfrak G(V)_{\ge 2}},(-)),$ where $\tT_{\mathfrak
G(V)_{\ge 2}}$ is the product of elements of $V$ of length $\ge 2,$
and $(-)$ is as in Theorem~\ref{ge22}.
 \end{defn}



\subsection{Symmetrization and the twist action}\label{symm}$ $

There is a general way to provide a negation map for arbitrary
Grassmann semialgebras. Although $\tT$-modules initially may lack
negation, one can obtain negation maps for them through the next
main idea, the symmetrization process, which although a special case
of super-semialgebras and their modules, provides a crucial method
of creating a triple. Tropical symmetrization dates back to
\cite{Gau}, and we recall the treatment for systems from
\cite{JuR1}.

 \begin{definition}\label{sym00}
Given any $\tT$-monoid module $\mathcal M$,   define its $\mathbb
Z_2$-graded \textbf{symmetrization} $\widehat {\mathcal M} =
\mathcal M \times \mathcal M$, with componentwise addition.

Also define $\widehat {\tT} =(\tT \times \{ \zero \}) \cup ( \{
\zero \} \times \tT)$ with  the \textbf{twist action} of $\widehat
\tT $ on $\widehat{\mathcal M}$
 given by the super-action, namely
 \begin{equation}\label{twi} (a_0,a_1)\ctw (c_0,c_1) =
 (a_0c_0 + a_1 c_1, a_0 c_1 + a_1 c_0), \quad a_i \in \tT, c_i\in  \mathcal M.  \end{equation}
\end{definition}

 \begin{defn}\label{symmod}  The \textbf{switch map}  $(-)_{\operatorname{sw}}$ on the
 symmetrized module
 $\widehat{\mathcal M}$ is given by
 $ (-)_{\operatorname{sw}}(c_0,c_1)=  (c_1,c_0).$

 If $\mathcal A$ is a semiring containing $\tT$,
 then we define the \textbf{twist action} as in \eqref{twi}, but
 this time with $a_i, c_i\in  \mathcal A.$
\end{defn}

\begin{thm}\label{gr1} (\cite[Theorems~2.41, 2.43]{JMR}) For any $\tT$-module $\mathcal A$, we can embed  $\mathcal
 A$ into $\widehat {\mathcal A}$ via $$\label{sum1} b \mapsto (b,\zero),$$ thereby
 obtaining a faithful functor from the category of semirings into  the category of
 semirings with a negation map (and preserving additive
 idempotence). This also yields a faithful functor from ordered semigroups to  signed $(-)_{\operatorname{sw}}$-bipotent
 systems. Any $\mathcal A$-module $\mathcal M$ yields a $\widehat{\mathcal A}$-module $\widehat{\mathcal M}= \mathcal M \oplus \mathcal M,$
 which has a signed
 decomposition   where $\mathcal M^+$ is the
 first component.\end{thm}

This applies to the Grassmann semialgebra:

 \begin{thm}\label{gr2}  Define $\mathfrak G(V)^\natural$ to be  $ \widehat{\mathfrak
G(V)}$ modded out by the congruence generated by $(v \w v',0) \cong
(0, v' \w v)$ for all $v,v' \in V$,  and $\overline{\tT}$ to be the
corresponding image of $\widehat {\tT}$. There is a triple
$(\mathfrak G(V)^\natural,\widehat {\tT},(-)_{\operatorname{sw}})$,
 together with an embedding of triples    $$ (\mathfrak G(V)_{\ge
2},\tT_{\mathfrak G(V)_{\ge 2}},(-)) \to (\mathfrak G(V)^\natural,
\overline{\tT},(-)_{\operatorname{sw}})$$  given by  $c \mapsto
(c,0)$.
\end{thm}

\begin{proof} Take ${\mathcal A}= \mathfrak
G(V)$ in Theorem~\ref{gr1}. $$(-)( v \w v')  = v' \w v \mapsto (v'
\w v, 0) = (-)_{\operatorname{sw}}(0, v'\w v) =
(-)_{\operatorname{sw}}(v\w v' ,0 ).$$ Thus $(-)$ matches by
\eqref{sum1}, so the Grassmann relations match.
\end{proof}
%
%
%
%
%

\subsection{The partially reduced Grassmann
system (when  $V$ is free)}$ $

The main results of this paper involve the \textbf{free module} $V$
 with base $\mathcal B = \{ b_0, \dots, b_{n-1}\},$ in the sense that any element of $V$
can be written uniquely as an $\mathcal  A$-linear combination of
the $b_i$.
 Let $V_n:=\Acal^{(n)}$    be the free module over the semiring $\Acal$ with basis $\bfb:=\{b_0,\ldots,b_{n-1}\}$ of $n$ elements.
  When  $V= V_n$,  this includes the definition  in
\cite[Definition~3.1.2]{GG}, in which $(-)$ is the identity map. In
this work we have two candidates for the Grassmann algebra, given
respectively in Theorems~\ref{gr2} and  \ref{gr3}.

\begin{rem}$ $ \begin{enumerate}\eroman
\item By  Theorem~\ref{impl}, $b_{\pi(i_1)}\w \cdots \w b_{\pi(i_t)}= (-)^{\pi} b_{i_1}\w
\cdots \w b_{i_t}$  for any permutation $\pi.$ Also,   every simple
tensor in which some $b_i$ repeats is $\zero$.

\item ${\mathfrak G}_k$ is free with a base of  $2{\binom n k}$ elements. For instance ${\mathfrak G}_2(V_3)$ has  base $b_1\w b_2$, $b_1\w b_3$, $b_2\w
b_3$ and their ``negations'' $b_2\w b_1$, $b_3\w b_1$, $b_3\w b_2$.
This phenomenon gives rise to the ``eigenvalue pair'' of \S\ref{eip}
below.
\end{enumerate}
\end{rem}

\begin{lem}\label{biG11} For the free Grassmann semialgebra,
  $\mathfrak G = \mathfrak G_{\operatorname{even}} \oplus \mathfrak G_{\operatorname{odd}}$ is a
  super-semialgebra, and its ideal $\mathfrak G^{\ge 2} = \mathfrak G_{\operatorname{even}}^{\ge 2} \oplus \mathfrak G_{\operatorname{odd}}^{\ge 2}$
 has the negation map from Theorem~\ref{ge22}. \end{lem}
\begin{proof}  By   linearity, we need only check products of the
$b_{i}$.
\end{proof}

\begin{lem}\label{wd} $(-)$  is well-defined, and
 \[\bar
b_{\pi(i_1)}\wedge \cdots  \wedge \bar  b_{\pi(i_t)} = (-)^{\pi}
\bar b_{i_1}\wedge \cdots \wedge \bar  b_{i_t}, \quad \forall t\ge
2.\]
\end{lem}
\begin{proof} $(-)$  is well-defined  by Theorem~\ref{ge22}. The
formula follows from writing a permutation as the product of
transpositions,   noting that the sign of a permutation is
well-defined, and counting the number of times $(-)$ occurs.
\end{proof}

\begin{lem}\label{biG2} Suppose $\sum _{\mathbf i} \al _{\mathbf i} b_{i_1} \otimes \cdots \otimes b_{i_k}+d \preceq  \sum _{\mathbf i} \al _{\mathbf i'} b_{i_1}' \otimes \cdots \otimes b_{i_k}'+d'  ,$ where
$i_1 < \dots < i_k$, $\al _{\mathbf i} , \al _{\mathbf i'}  \in
\mathcal A,$ $d,d' \in \overline{\mathfrak G ^\circ}.$ Then $\sum
_{\mathbf i} \al _{\mathbf i} b_{i_1} \otimes \cdots \otimes b_{i_k}
\preceq_\circ  \sum _{\mathbf i} \al _{\mathbf i'} b_{i_1}' \otimes
\cdots \otimes b_{i_k}'  ,$ where $i_1 < \dots < i_k$.
\end{lem}
\begin{proof} Match components, eliminating those components in which some $b_i$ repeats or some of the $b_i$ descend.
\end{proof}

\begin{defn}
  $ \Gdbl  $ is the ideal of
$T(V) $ generated by  all elements $b_i \otimes b_i$ for all $i.$
\end{defn}


\begin{lem}\label{biG1} Any nonzero element of $T(V)$ is a sum of
terms $(\pm) \a\, b_{i_1}  \cdots \otimes b_{i_k}+d,$ where $i_1 <
\dots < i_k$, $\al \in \mathcal A,$ and $d \in \Gdbl.$
\end{lem}
\begin{proof}  We rearrange the $b_i$ appearing in the summands, noting
that any time a $b_i$ repeats, the product is in $ \Gdbl.$
\end{proof}

We use Lemma~\ref{biG2} to avoid $ \Gdbl$ in our computations.
 We will need the
following nondegeneracy result.

\begin{prop} \label{nondegen} Suppose $V = \mathcal A ^{(n)}$ and $u,u'\in \mathfrak
G(V)_k$ for $2 \le k <n.$ \begin{enumerate}\eroman
\item If $u \w v = u' \w v$ for all $v \in \mathfrak G(V)_{n-k}$,
then $u= u'$.
\item
If $u\notin {T(V)_k^\circ}$ then there is some $v \in T(V)_{n-k}$
for which $u\w v\notin {T(V)_k^\circ}$.
\end{enumerate}\end{prop}
\begin{proof} Using Lemma \ref{biG1}, write $u = \sum _{i_1< \dots < i_k}\a_{i_1, \dots i_k} b_{i_1} \w
\cdots \w b_{i_k}$, $u' = \sum _{i_1< \dots i_k}\a_{i_1, \dots i_k}'
b_{i_1'} \w \cdots \w b_{i_k'}$.

 (i) For any $\a_{1, \dots k} \ne \zero,$   $u \w b_{i_{k+1}}\w
\cdots \w b_{i_n} = (\pm)\a_{1, \dots k} b_1 \w \cdots \w b_n,$
which must be
 $(\pm)\a'_{1, \dots k} b_1 \w \cdots \w b_n,$ with the base elements
 matching up.

(ii) Adjusting notation, we may assume that $\a_{1, \dots k}\ne 0$.
But then $$u \w b_{k+1}\w \cdots \w b_n = \a_{1, \dots k} b_1 \w
\cdots \w b_n \notin {T(V)_n^\circ}^\sharp.$$
\end{proof}

\begin{thm}\label{gr3}  More in line with~\cite{GG},
define the \textbf{partially reduced Grassmann algebra} $\mathfrak
G(V)_{\ge 2}^\diamondsuit$ to be  $ T(V)_{\ge 2}$ modded out by the
congruence generated by $(b_i \w b_i,0) $ for all $i,$
  $\mathfrak G(V)^\sharp$ to be  $ \widehat{T(V)}$ modded out by the
congruence generated by $(b_i \w b_j,0) \cong (0, b_j \w b_i)$ for
all $i,j$, and $\widehat {\tT}^\sharp$ to be the corresponding image
of $\widehat {\tT}$.

 There is a triple $(\mathfrak G(V)^\sharp,\widehat
 {\tT}^\sharp,(-)_{\operatorname{sw}})$,
 together with an embedding of triples    $$ (\mathfrak G(V)_{\ge
2}^\diamondsuit,\tT_{\mathfrak G(V)_{\ge 2^\diamondsuit}},(-)) \to (
\mathfrak G(V)^\sharp, { \tT}^\sharp,(-)_{\operatorname{sw}})$$
given by  $c \mapsto (c,0)$.
\end{thm}

\begin{note} $\mathfrak G(V)_{\ge 2}$ is clearly a homomorphic image of $\mathfrak G(V)_{\ge
2}^\diamondsuit$, since it has more relations. (All $v\otimes v$ are
sent to $\zero$, not just the $b_i \otimes b_i$.) Although
$\mathfrak G(V)_{\ge 2}^\diamondsuit$ and $\mathfrak G(V)_{\ge 2}$
both coincide with the Grassmann algebra in the classical setting
where $V$ is a vector space over a field, they differ in the
semiring setting, since $b_i \otimes b_j + b_j \otimes b_i$ is not
$\zero$ in $\mathfrak G(V)_{\ge 2}^\sharp$.\end{note}

\subsection{Digression: Related notions}$ $

For the remainder of this section we examine algebraic notions
related to this paper, even though one can bypass them for the
proofs of Theorems~\ref{prech} and \ref{CHthm}.

\subsubsection{The case when $V$ already  has a negation map}$ $

 We have
seen that  $V$ itself need not have a negation map, for us
``almost'' to define a negation map on $T(V)$. In case $V$ does have
a negation map $(-)$\footnote{ For example $V$ could be the free
$\mathcal A$-module with a negation map, with base $\{ b_i, (-)b_i :
i \in I\}$.}, we need a slight modification. We define a negation
map on the tensor product $V \otimes  W$ by $(-)(v \otimes w) = ((-)
v) \otimes w.$ When $W$ also has a negation map $(-)$
 we
define a \textbf{negated tensor product}  $V \otimes _{(-)} W$ by
imposing the extra axiom
$$((-)v) \otimes_{(-)} w = v \otimes _{(-)} ((-)w), \qquad v\in V, w \in W.$$ (One
mods out   the tensor product  by the congruence generated by all
elements $((-)v \otimes w,\ v \otimes (-)w)$ .)

\begin{rem}\label{Grasstrip}
The appropriate  triple  is $(\mathfrak G, \tT_{\mathfrak G}, (-))$,
where $\tT _{\mathfrak G}= \{ v_1 \wedge\cdots \wedge v_t:\ v_i \in
V,\ t \in \Net \},$ the submonoid generated by $\tT$, with $(-)(v_1
\wedge\cdots \wedge v_t) = ((-) v_1) \wedge\cdots \wedge v_t.$
\end{rem}

\subsubsection{Comparison with \cite{GG}}$ $

 The  arguments of \cite{GG}, whose objective is to obtain a Grassmann algebra point of view for Pl\"{u}cker relations,
 can be adapted to this situation. We use the systemic version which
 enables us to replace the ``bend relation'' $f \sim g$ of  \cite{GG} for $f,g \in \Hom (V,\mathcal A)$
 by $f+g$ being a quasi-zero in the sense that $(f+g)(v) \in \mathcal A
 ^\circ$ for every $v\in V.$ Then the  Pl\"{u}cker relations in \cite[Proposition~4.1.2]{GG}
become the conditions that $\sum _{i \in A \setminus B} v_{A-\{i\}}
v_{B+\{i\}}$ is a quasi-zero.
 \begin{note} The flavor of the Grassmann algebra might be better preserved by
taking the negation map $(-)$  not to be the identity map, but
rather as defined here, which also could be obtained using
symmetrization. Note also that $\mathfrak G$ is commutative in
\cite[Definition~3.1.2]{GG}. So  why does
\cite[Definition~3.1.2]{GG} work? The answer is that $\mathfrak G$
is largely a book-keeping devise to keep track of sets of vectors
without repetition, and application of its theory to matroids does
not require much of multiplication other than $e_i \wedge e_i = 0$.
\cite[Proposition~3.1.4]{GG} is formal. One needs a cancelation
result parallel to \cite[Lemma~3.2.2]{GG}, and
\cite[Proposition~4.2.1]{GG} requires the ability to switch vectors
$e_i$ and $e_j$. 
\end{note}

%
%
%

\begin{rem}\label{zerosum}  Suppose that $\mathcal  A$ is ``zero sum free'' in the sense that $a_1+a_2 = \zero$ implies $a_1=a_2 = \zero$. Then the base
$\mathcal B$ of a free module $V$ is unique
 up to multiplication of invertible elements of $\mathcal  A$.
 (Otherwise some $b_i$ does not appear in the new base, and we cannot
 recover $b_i$ since we cannot zero out extraneous coefficients.\end{rem}

\subsubsection{Digression: The Grassmann envelope}$ $

\begin{rem}\label{Grasstrip1}  Just as with classical
algebra, one can use   $\mathfrak G $ to study a super-semialgebra
$\mathcal A = \mathcal A_0 \oplus \mathcal A_1$ by defining its
\textbf{ Grassmann envelope} $\mathcal A_0 \otimes \mathfrak G _0 +
\mathcal A_1 \otimes \mathfrak G_1 \subset \mathcal A \otimes
\mathfrak G .$ Following Zelmanov, we say that a super-semialgebra
$\mathcal A$ is \textbf{super-P} if its Grassmann envelope is P. For
example, $\mathcal A$ is super-commutative if its Grassmann envelope
is commutative. In particular, $\mathfrak G $ itself is
super-commutative.

\end{rem}

Then one can study linear algebra over super-commutative
super-semialgebras, super-anticommutative super-semialgebras, and so
forth, as indicated in \cite[\S 8.2.2]{Row16}.

   \section{Hasse-Schmidt Derivations on Grassmann Semi-Algebras}\label{basics}

Having set out the general framework, let us turn to the situation
at hand.  We review our set-up, in the special case of power series
over endomorphisms of the Grassmann algebra.
 As before, $V_n:=\Acal^{(n)}$   is the free module   with basis $\bfb:=\{b_0,\ldots,b_{n-1}\}$.
   (We start our subscripts with~0 in consonance with the notation for
   projective space.)
   Let $T_0 (V_n) = \Acal$, and $T_k(V_n):=V_n\otimes V_n\otimes \dots \otimes V_n$ be its $k$ tensor power. Define a negation $(-):T_2(V_n)\sra T_2(V_n)$
   by mapping $u\otimes v$ to $v\otimes u$.
  In particular $(-)(u\otimes u)=u\otimes u$.  We extend this to  $(-):T_k(V_n)\sra
   T_k(V_n)$ by means of~Theorem~\ref{ge22} and Lemma~\ref{wd}. Let
  \[
  T_{\geq 2}(V_n)= \{\zero\} \cup \, \bigoplus_{k\geq 2}T_k(V_n),
  \]
 a \semiringg0 with multiplication  given by  tensoring.
   We  consider two variants of the Grassmann algebra:
 \begin{enumerate}
\item $\mathfrak
G(V)_{\ge 2}^\sharp$, modding
  out by   the congruence~$\mathcal I$ of $T_{\geq 2}(V_n)$ generated by
   all $\{(b_i\otimes b_i,0): 0 \le i <n\},$

  \item $\mathfrak
G(V)_{\ge 2}$, modding
  out by   the congruence~$\mathcal I$ of $T_{\geq 2}(V_n)$ generated by
   all $\{(u\otimes u,0): u\in V_n\}.$
\end{enumerate}

%
In either case we will work with a graded Grassmann semialgebra
$\mathfrak G$, which now we denote as

\[
{\wV}=\bigoplus_{r\geq 0}\wrV, \qquad \text{where}
\]

\[
\bw^0V_n=\Acal,\qquad \bw^1V_n=V_n,\qquad
\mathrm{and}\quad\wrV:=\frac{T_r(V_n)}{ \mathcal I\cap
T_r(V_n)}\quad for \,\, r\geq 2.
\]

Thus $u\w v$ denotes the image of $u\otimes v$ through the natural
map $T(V_n)\sra \wV$. 
Here $\preceq$ is $\preceq_\circ.$

\begin{rem}\label{Grassneg} By Theorem~\ref{ge22}, each submodule $\bw^rV_n$, $r\geq 2,$
inherits a negation map by putting
$$
(-)(u_1\w u_2\cdots\w u_r)=u_2\w u_1\w \cdots\w u_r.
$$
\end{rem}

\begin{rem}\label{Grasstrip2}$ $
\begin{enumerate}\eroman
\item For each $r\geq 2$,  $\wrV$ is spanned by words $b_{i_0}\w
b_{i_1}\w \cdots\w b_{i_{r-1}}$ of length $r$. 
 In particular, $\wrV$ is a
free $\Acal$ module spanned by $\wb^r_\blamb$, where

\[ \blamb:=(\lambda_1\geq \cdots\geq \lambda_r), \qquad
\wb^r_\blamb:=b_{\lambda_r}\w b_{1+\lambda_{r-1}}\w\cdots\w
b_{r-1+\lambda_1}.
\]
\end{enumerate}
\end{rem}

 We are interested in the $\Net$-graded power series semiring
$(\wV)[[z]]: = \oplus _{r \ge 0} \wV z^r $ of Definition~\ref{psr}
(and later its super-version), and its endomorphisms.

Since the congruences are homogeneous, we define \be
{\bigwedge}^{\ge 1}V_n:=\bigoplus_{r\geq 1}\bw^rV_n, \qquad
\qquad{\bigwedge}^{\ge 2}V_n:=\bigoplus_{r\geq 2}\bw^rV_n, \qquad
\mathrm{and} \qquad{\bigwedge}^{\ne 1}V_n:=\bigoplus_{r\ne 1
}\bw^rV_n \ee

\begin{definition}
Let  $\Dz :=\sum_{i\geq 0}D_iz^i\in \mathrm{ End}(\wV)[[z]]$ be
homogeneous of degree $0$ (i.e. $D_i(\wrV)\subseteq \wrV$ and in
particular $D_i(V_n) \subseteq V_n)$ ). If \be \Dz (u\w v)=\Dz  u\w
\Dz  v\label{eq:HSderG} \ee we say that it is a {\em Hasse-Schmidt
(HS) derivation} on $\wV$.
\end{definition}

To simplify notation let us simply denote the identity map on $V_n$
as ``$\one_V$,'' also identified with $D{\zparen}$ where $D_0 =
\one$ and all other $D_i  = \zero.$

Equation (\ref{eq:HSderG}) is equivalent to:

\be \label{HS1} D_k(u\w v)= \sum_{i+j=k}D_iu\w D_jv,\quad \forall
k\geq 0, \quad \forall u,v\in \wV.\ee

For $r\geq 2$, any element of $\wrV$ is a linear combination of
monomials $v_1\w\cdots\w v_r$ of length $r$. The definition shows
that $\Dz $ is uniquely determined by the values it takes on
elements of  $V$.

In the following we shall restrict to a special class of $HS$
derivations, useful for the applications.
\begin{proposition} \label{indder} {\em For any $f\in End_{\Acal}(V_n)$, there exists a unique $HS$-derivation $\Dfz$ on $\wV$ such that}
$\Dfz_{|_{V_n}}=\sum_{i\geq 0}f^iz^i.$
\end{proposition}
\proof For the chosen $\Acal$-basis of  the module $V$  we
necessarily have $\Dfz(b_j)=\sum_{i\geq 0}f^i(b_j)z^i$. Write $f(z)$
for $\sum_{i\geq 0}f^iz^i.$ One defines $\Dfz$ on  $\wM$ by setting
for each degree: \be \Dfz(b_{i_1}\w\cdots\w
b_{i_j})=f(z)b_{i_1}\w\cdots\w f(z)b_{i_j},\ \ 1\leq j\leq r. \ee If
$D$ were another derivation satisfying the same initial condition,
it would coincide on all the basis elements of $\wV$, which generate
all elements of $\wV$. \qed
\begin{example}
Let us compute $D^f_2(b_1\w b_2)$ where $f(b_i)=b_{i+1}$. Then
\[\begin{aligned}
D^f_2(b_1\w b_2) & =D^f_2(b_1)\w b_2+D^f_1b_1\w D^f_1b_2+b_1\w
D^f_2b_2\\&  =f^2(b_1)\w b_2+f(b_1)\w f(b_2)+b_1\w f^2(b_2)
\\&=b_3 \w b_2+b_2\w b_3+b_1\w b_4\succeq b_1\w b_4,
\end{aligned} \]
since $b_3 \w b_2+b_2\w b_3$ is a quasi-zero.
\end{example}

From now on we shall fix the endomorphism $f$  once and for all, and
write $\Dz :=D^f(z)$ and $D:={D_1}_{|V}:=f$. Also we write $D_iv$
for $D_i(v)$ and $\Dz v$ for $\sum D_i v\cdot z^i$. In particular, for
each $v\in V_n$ the equality $D_iv=D_1^iv=f^i(v)$ holds.

\begin{lem}\label{Dform} For $u,v \in V_n$, \begin{enumerate}\label{distr31}\eroman
  \item $\Dz v = v+  \Dz (D_1 v) z$.
   \item $\Dz (u\w v) = u\w \Dz v +z \Dz  (D_1 u \w v).$

\end{enumerate}
 \end{lem}
\begin{proof} (i) $\Dz v = v + \sum _{i\ge 1} D_i v\, z^i = v + \sum _{i\ge 1} (D_{i-1} D_1 v
\, z^{i-1})z= v+  \Dz (D_1 v) z.$\

\bigskip

$ \label{try1} \begin{aligned} \text{(ii)}\ \Dz (u\w v) & = \Dz
u\w\Dz v = ( u+ z\Dz  D_1 u )\w\Dz  v) \\& = u\w \Dz v +z\Dz D_1 u \w \Dz v =  u\w \Dz v  +z\Dz (D_1u \w
v).\end{aligned}$
\end{proof}

\subsection{The canonical quasi-inverse of $\Dz $}\label{sec61}
\begin{defn}
 $\overline {D} \{z\} :=\sum_{i\geq 0}\ovD_iz^i\in \End(\bw^{\neq
1}V_n)[[z]]$  is a \textbf{(left) quasi-inverse} of $\Dz $ if \be
\overline {D} \{z\} \Dz u\succeq u, \quad \forall u\in \bw^{\neq
1}V_n\label{eq:qivs} .\ee
\end{defn}
Our next task consists in constructing a quasi-inverse $\overline
{D} \{z\} $ of the $HS$ derivation $\Dz $, that we will achieve
through a number of steps necessary to cope with the difficulty of
not having a natural negation map on $V_n$. This can be done in two
ways: First do it in the classical case, and then apply
the ``transfer principle'' of Remark~\ref{trans}. 
However, one gets more precise information by taking the direct
analog.

\bigskip

\noindent {\bf Construction.} Towards this purpose we first consider
the map: \be \left\{\begin{matrix}\ovD&:& V_n&\lra&\End_\Acal(\bw
V_n)\\ \\ &&u&\longmapsto&\ovD u
\end{matrix}\right.
\ee such that $\ovD u(v):=v\w D_1u$ for all $v\in V_n$. If $v\in
\bw^{\geq 2}V_n$, we may assume it is of the form $v=v_1\w v_2$ with
$v_1\in V_n$. In this case we define \be \ovD u(v)=\ovD u(v_1\w
v_2)=\ovD u(v_1)\w v_2=v_1\w D_1u\w v_2.\label{eq36:36} \ee It is
easily seen that \eqref{eq36:36} suffices to define $\ovD u$ on all
the Grassmann semi-algebras and also that the definition does not
depend on the representation of the same element $v\in \bw V_n$. In
fact any one such is a finite linear combination of tensors of the
form  $v:=v_1\w\cdots\w v_k$, and the definition of $\ovD u(v)$ does
not change if we replace the given expression of $v$ with an
equivalent one after an even permutation of the factors. It will be
useful to identify $V_n$ as a subset of $\End_\Acal (\bw V_n)$, by
viewing each of its elements  as a (wedge) multiplication operator,
under the map $v\mapsto v\w _{\underline{\phantom{w}}}$, i.e.
$v(w)=v\w w$ for all $w\in\bw V_n$.
\begin{definition}\label{def:def38}
Let $\ovD\{z\}:=1+\ovD_1z+\ovD_2z^2+\cdots+\ovD_nz^n:V_n\sra \End_\Acal(\bw V_n)[z]$ defined as follows. If $u_1\w \cdots\w u_k\in \bw^kV_n$, then
\be
\ovD\{z\}(u_1\w\cdots\w u_k)=\ovD\{z\}(u_1)\circ\cdots\circ \ovD\{z\}(u_k)
\ee
where $\circ$ is the composition in $\End_\Acal(\bw V_n)$ and
where for all $u\in V_n=\bw^1V_n$ we set
\be
(\ovD\{z\}u)(v)= u\w v+z\ovD u(v) \in \End_\Acal(\bw V_n)[z]
\ee
acting on $v\in\bw V_n$ as $\ovD\{z\}u (v)=u\w v+(\ovD u)v$.
Now we extend $\ovD\{z\}$ to a map $\bw V_n\sra \End_\Acal(\bw V_n)$ defining it on monomials of degree $k\geq 2$ via
$$
(1+\ovD_1z+\cdots+\ovD_kz^k)(u_1\w\cdots \w
u_k)=\ovD\{z\}(u_1\w\cdots\w u_k)=\ovD\{z\}(u_1)\circ\cdots\circ
\ovD\{z\}(u_k).
$$

\end{definition}

\begin{remark}\label{rem3.9}
By definition it follows that if $u\in\bw^iV_n$, then $\ovD_ju=0$
for all $j>i$.
\end{remark}

\begin{example} In this example we compute $\ovD_1(u_1\w u_2)$, $\ovD_2(u_1\w u_2)$ and $\ovD_2(u_1\w u_2\w u_3)$ to illustrate how the definition works.
By definition $\ovD_1(u_1\w u_2)$ and $\ovD_2(u_1\w u_2)$ are the coefficients of $z$ and $z^2$ in the expansion $(\ovD u_1)\circ (\ovD_1 u_2)$ that we apply to a test element $w\in \bw V_n$. By definition one may assume $w\in V_n$. One has:
\begin{eqnarray*}
&&(u_1\w u_2+ z(\ovD u_1\circ u_2+u_1\circ \ovD u_2)+z^2\ovD u_1\circ \ovD u_2)w
\\
&=& u_1\w u_2\w w+z \ovD u_1(u_2\w w)+u_1\w \ovD u_2(w)+z^2\ovD u_1(\ovD u_2(w))\\
&=&u_1\w u_2\w w+z(u_2\w D_1u_1\w w+ u_1\w (w \w D_1u_2))+z^2\ovD u_1(w\w D_1u_2)\\
&=& u_1\w u_2\w w+z(u_2\w D_1u_1+D_1u_2\w u_1)\w w+z^2(w\w D_1u_1\w D_1u_2)=\\
&=&[u_1\w u_2+(D_1u_2\w u_1+u_2\w D_1 u_1)z+(D_1u_1\w D_1u_2)z^2]\w
w.
\end{eqnarray*}

We have obtained:
$$
\ovD_1(u_1\w u_2)= (D_1u_2\w u_1+u_2\w
D_1u_1)\w_{\underline{\phantom{w}}} \,\,\,:\bw V_n\sra \bw V_n,
$$
$$
\ovD_2(u_1\w u_2)=(D_1u_1\w
D_1u_2)\w_{\underline{\phantom{w}}}\,\,\,: \bw V_n\sra \bw V_n.
$$
Similarly we can find that
$$
\ovD_2(u_1\w u_2\w u_3)=\ovD_2(u_1\w u_2)\circ u_3+\ovD_1(u_1\w u_2)\circ \ovD_1 u_3,
$$
i.e., more explicitly
$$
\ovD_2(u_1\w u_2\w u_3)\w w= D_1u_1\w D_1u_2\w u_3\w w+(D_1u_2\w u_1+u_2\w D_1u_1)\w w \w D_1u_3
$$
from which
\begin{eqnarray*}
\ovD_2(u_1\w u_2\w u_3)\w w&=&(D_1u_1\w D_1u_2\w u_3+D_1u_2\w
D_1u_3\w u_1+ u_2\w D_1u_3\w D_1u_1)\w w\cr &=&(D_1u_1\w D_1u_2\w
u_3+u_1\w D_1u_2\w D_1u_3+D_1u_1\w u_2\w D_1u_3)\w w.
\end{eqnarray*}
We could say that the operators $\ovD_1(u_1\w u_2)$, $\ovD_2(u_1\w
u_2)$, and $\ovD_2(u_1\w u_2\w u_3)$ are ``represented''
respectively by the following elements of $\bw V_n$:
$$
D_1u_2\w u_1+u_2\w D_1u_1=D_1(u_2\w u_1),\quad \qquad D_1u_1\w
D_1u_2,
$$
and
$$
D_1u_1\w D_1u_2\w u_3+u_1\w D_1u_2\w D_1u_3+D_1u_1\w u_2\w D_1u_3.
$$
\end{example}

\begin{remark}
Let us check that
 $D_1u\w v+(\ovD_1u)(v)\succeq 0,$
 which is the sense we want to give to the expression $D_1+\ovD_1\succeq 0$.
 For all $u,v\in \bw V_n$. If $u\in V_n$ and $v=v_1\w v_2$ with $v_1\in V_n$:

\begin{eqnarray*}
D_1 u\w v+(\ovD_1 u)(v)&=&(D_1 u\w v_1)\w v_2+(\ovD u)(v_1)\w v_2\\
&=&D_1u\w v_1\w v_2+v_1\w D_1u\w v_2\\
&=& (D_1 u\w v_1+v_1\w D_1u)\w v_2\succeq 0.
\end{eqnarray*}
\end{remark}
More  generally we have the following crucial:
\begin{theorem}\label{prech1} The  polynomial $\ovDz$ is a quasi-inverse of $\Dz$, in the sense that for all $u\in \bw V_n$
\be
\ovDz  \Dz u\succeq u.
\label{eq:invovdd}
\ee
\end{theorem}
\proof We first check that the property holds for all $u\in V_n$. Then, for all $w\in \bw V_n$
\begin{eqnarray*}
(\ovDz \Dz u)(w)&=&(\Dz u+z\ovD\Dz u)(w)\\
&=&\Dz u\w w+zw\w D_1\Dz u\\
&=&(u+zD_1\Dz u)\w w+z\cdot w\w D_1\Dz u\\
&=&u\w w+z\big(D_1\Dz u\w w+w\w D_1\Dz u\big)\succeq u\w w
\end{eqnarray*}
for all $w\in \bw V_n$. Thus we have proved that $\ovDz\Dz u \w
\underline{\phantom{w}}\, \,\succeq \,  u\w
\underline{\phantom{w}}$, and the property is checked for all $u\in
\bw^1V_n$. Now,  we argue by induction, by supposing the property
holds true for all $u\in \bw^{\leq k-1}V_n$. Let us prove it for all
$u\in \bw^kV_n$. In this case we can assume $u$ of the form $u_1\w
v$, with $u_1\in V_n$. Then
$$
\ovDz(\Dz(u_1\w v))=\ovDz(\Dz u_1)\circ\ovDz(\ovDz v)\succeq (u_1\w
v)\w \underline{\phantom{w}}
$$
having used induction and the first step. \qed

\begin{cexample} Quite surprisingly, while $\ovDz$ is a quasi-inverse of $\Dz$, the
reverse is not true. For instance, for all $u,v\in V_n$ one has: $
\Dz(\ovDz u)(v)\succeq u\w \Dz v=u+\Dz D_1u. $ Let us check it,
recalling that if $u\in V_n$ then $\Dz u=u+\sum_{i\geq 0}D_1^iu\cdot
z^i$.
\begin{center}
\begin{tabular}{rllrr}
$\Dz(\ovDz u)(v)$&$=$&$\Dz (u\w v+ z\cdot v\w
D_1u)$&&{\em(}definition \eqref{eq36:36} of $(\Dz u)(v)${\em)}\\
&$=$&$\Dz u\w \Dz v+z\Dz v\w\Dz D_1u$&&{\em(}$\Dz$ is a HS
derivation)\\ &$=$&$u\w \Dz v+z\Dz D_1 u\w \Dz v$\cr &$+$&$z\Dz v\w
\Dz D_1u$&&(using  $\Dz u =u+\Dz D_1u)$\\ &$\succeq$&$u\w \Dz v.$
\end{tabular}
\end{center}
As an additional check, notice that if  $\Dz$ were a quasi-inverse
of $\ovDz$, then
$$
(1+D_1z+D_2z^2+\cdot)(1+\ovD_1 +\ovD_2z^2+\cdots)\succeq 0
$$
In particular, considering the coefficient of $z^2$ in both sides,
the following surpassing relation should hold: \be
D_2+D_1\ovD_1+\ovD_2\succeq 0.\label{eq:counsur} \ee But
\eqref{eq:counsur}  already fails for  $u\in V_n=\bw^1V_n$. Indeed,
for all $v\in V_n$, and noting that $\ovD_2u=0$, by~
Remark~\ref{rem3.9}:
\begin{center}
\begin{tabular}{rllrr}
$(D_2u+D_1\ovD_1u+\ovD_2u)(v)$&$=$&$D_2u\w v+D_1(v\w
D_1u)$&&{\em(}applying to a test vector $v${\em)}\\ &$=$&$D_2u\w
v+D_1v\w D_1u+v\w D_2u$&&{\em(}Leibniz rule enjoyed by $D_1$)\\
&$=$&$D_1v\w D_1u+ D_2u\w v+v\w D_2u$\\ &$\succeq$&$D_1v\w D_1u$
\end{tabular}
\end{center}

\end{cexample}

\begin{theorem}\label{prech}
$ \overline {D} \{z\} (\Dz  u\w v)\succeq u\w \overline {D} \{z\} v,
\qquad \forall u,v\in\bw V_n. \label{eq:prech} $
\end{theorem}
\proof Suppose $u,v$ are homogeneous, say $u\in \bw^kV_n$ and $v\in
\bw^\ell V_n$. Then, by Definition \ref{def:def38} of $\ovDz$,
$$
\ovDz(\Dz u\w v)=\ovDz(\Dz u)\circ\ovDz v
$$
By   Theorem \ref{prech1}, $ \ovDz(\Dz u\w v)(w)\succeq u\circ(\ovDz
v)(w)=u\w (\ovDz v)(w) ,$ because $u$ acts  as an endomorphism on
vectors of $\bw V_n$  as $u\w \underline{\phantom{w}}$, for all
$w\in\bw V_n$. \qed

\subsection{The Cayley-Hamilton formulas for semialgebras}\label{eip}$ $


Formally define $\zeta = b_0\w b_1\w\cdots\w b_{n-1}$ and $ \zeta' =
b_1\w b_0\w\cdots\w b_{n-1}.$ Thus $\zeta' = (-)\zeta,$ and
\be\label{eqa} \ovD_i\zeta = e_i\zeta + e_i'\zeta', \quad
e_i,e_i'\in \Acal. \ee

In other words, $(e_i,e_i')$ could be called the  \textbf{eigenvalue
pair} of $\ovD_i$ restricted to $\bw^nV_n$ (where in some sense
$e_i'$ is the negated part). Let $E_n(z)$ be the eigenvalue
polynomial of $\overline {D} \{z\} $, i.e.
\[
E_{n}(z)\zeta:=\overline {D} \{z\} \zeta + \overline {D} \{z\}
\zeta'=(1+e_1z+\cdots+e_nz^n)\zeta + (1+e'_1z+\cdots+e'_nz^n)\zeta'.
\]
In particular if one sets $D_i\zeta=h_i\zeta + h_i'\zeta'$, the
relations $\overline {D} \{z\} \Dz \zeta\succeq \zeta$ and
$\overline {D} \{z\} \Dz \zeta'\succeq \zeta'$ yield  the relation
\be \label{eq:precheh} (h_n+e_1h_{n-1}+\cdots+e_n) +
(h'_n+e'_1h'_{n-1}+\cdots+e'_n) \succeq \zero. \ee
\begin{theorem}\label{CHthm}
 The Cayley-Hamilton formulas  \be \label{eq:CHform}
\left((D_nu+e_1D_{n-1}u+\cdots+e_{n}u)\w v\right)
(-)\left((D_nu+e'_1D_{n-1}u+\cdots+e'_{n}u)\w v \right) \succeq
\zero
 \ee hold for all $u\in \bw^{>0}V_n$, i.e., the left side is a
quasi-zero.
\end{theorem}
\proof
 If $u=\zeta$ the theorem is true, due to (\ref{eq:precheh}). Then assume that $u\in\bw^{n-i}V_n$, for some $1\leq i\leq n-1$.
This follows from the transfer principle of Remark~\ref{trans},
since the assertion was proved (with equality) for classical
algebras in \cite{GaSc2}, and all the extra quasi-zeros appear in
the right. But we also would like to give a direct proof. For all
$v\in \bw^{i}V_n$ we have the surpassing relation~(\ref{eq:prech}).
Matching degrees yields the surpassing relation between the $n$-th
degree coefficient of the left side and the $n$-th degree
coefficient of the right side of~\eqref{eq:prech} which is:
$$
D_nu\w v+\ovD_1(D_{n-1}u\w v)+\cdots+\ovD_n(u\w v)\succeq u\w
\ovD_nv.
$$
Since $\overline {D} \{z\}  v$ is a polynomial of degree at most
$i<n$, it follows that $\ovD_kv \succeq \zero$ for all $k > i$.
 On the other hand
$\ovD_i(D_{n-i}u\w v)= e_i(D_{n-i}u\w v) (-)e_i'(D_{n-i}u\w v)$
because $(e_i, e_i')$ is the  eigenvalue pair of $\ovD_i$ against
any element of $\bw^nV_n\cong (\Acal\zeta+ \Acal\zeta')$. Thus we
have proved \eqref{eq:CHform} for all $v\in \bw V_n$.\qed

\begin{rem}\label{eq:CHform17} In the classical case where the only quasi-zero is $\{\zero
\}$, we get

\be \label{eq:CHform1} \left((D_nu+e_1D_{n-1}u+\cdots+e_{n}u)\w
v\right) (-)\left((D_nu+e'_1D_{n-1}u+\cdots+e'_{n}u)\w v \right) =
\zero.
 \ee
\end{rem}

\begin{corollary}\label{CHthm1}
$ (D_1^n+(e_1(-)e_1')D_1^{n-1}+\cdots+(e_n(-)e_n'))u\succeq \zero $
for all $u\in \bw^{>0} V_n,$ where we interpret
$(e_i(-)e_i')D_i^{n-i}(u)$ to be $e_i D_i^{n-i}u \, (-)\, e_i'
D_i^{n-i}u.$
\end{corollary}
\proof By Theorem~\ref{CHthm},
$$\left((D_nu+e_1D_{n-1}u+\cdots+e_{n}u)\w v\right)
(-)\left((D_nu+e'_1D_{n-1}u+\cdots+e'_{n}u)\w v \right) \succeq
\zero.$$ But $\Dz $ is by hypothesis the unique HS-derivation on
$\wV$ associated to the endomorphism $D_1$ (see
Proposition~\ref{indder}). In particular $D_iu=D_1^iu$.\qed

\begin{note}
 When  working with $\mathfrak
G(V)_{\ge 2}$, we obtain equality in Theorem \ref{CHthm} and
Corollary~\ref{CHthm1} since the only quasi-zeros are $\zero,$ by
Theorem~\ref{impl}.
\end{note}


\begin{thebibliography}{1}

\bibitem{AGG2}
M.~Akian,   S.~Gaubert, and A.~Guterman.
\newblock Tropical Cramer Determinants revisited,
\newblock {\em Contemp. Math.}  616, Amer. Math. Soc. 2014,  1--45.


%


 \bibitem{Berg} G.~Bergman, The diamond lemma for ring theory, Adv. in
Math. 29 (1978), no. 2, 178--218.




\bibitem{Ga2} L.~Gatto, Schubert calculus via Hasse-Schmidt derivations, Asian J. Math.
9 no. 3 (2005), 315--321.


\bibitem{GaSa} L.~Gatto and P~Salehyan,  Hasse-Schmidt derivations on
Grassmann algebras, IMPA monographs 4, Springer, 2016.


\bibitem{GaSc} L.~Gatto and I. Scherbak, On generalized Wronskians, in "Contributions in Algebraic Geometry",
Impanga Lecture Notes (P. Pragacz Ed.), EMS Congress Series Report,
257--296, 2012, available at http://xxx.lanl.gov/pdf/1310.4683.pdf.

\bibitem{GaSc2} L.~Gatto and I. Scherbak, Hasse-Schmidt derivations and Cayley-Hamilton theorem for exterior algebras,  Functional analysis and geometry: Selim Grigorievich Krein centennial,  Contemp. Math., {\bf 733}, Amer. Math. Soc., Providence, RI, 2019, 149--165,\href{https://arxiv.org/pdf/1510.03022.pdf}{arXiv:1510.03022}
(2015).


\bibitem{Gau} S.~Gaubert. \newblock {\em Th\'{e}orie des syst\`{e}mes lin\'{e}aires dans les
diodes.} Th\`{e}se, \'{E}cole des Mines de Paris, 1992.

\bibitem{GG} J.~Giansirancusa and N.~Giansirancusa, A Grassmann
algebra for matroids, Manuscripta Mathematica, 156 no. 1 (2018),
187--213.

\bibitem{golan92}
J.~Golan,
\newblock {\em The theory of semirings with applications in mathematics and
  theoretical computer science}, volume~54.
\newblock Longman Sci \&\ Tech., 1992.


 \bibitem{HaS}  H.
~Hasse
  and F.K.~Schmidt,
  Noch eine Begrundung der
  Theorie
  der
  hoheren
Differentialquotienten
 in
  einem
  algebraischen
  Funktionenkorper  einer  Unbestimmten,
J.~Reine Angew. Math.  177 (1936), 215--237.

 \bibitem{He1} N.~Heerema,
  Derivations
  and
 embeddings
 of
  a field
   in  its
  power series  ring, Proc.
 Amer. Math. Soc. 11 (1960), 188--194.

 \bibitem{He2} N.~Heerema, Derivations and embeddings of a field in its power series ring. II.
Michigan Math. J. 8 (1961), no. 2, 129--134.

\bibitem{He3} N.~Heerema, Higher derivations and automorphisms of complete local
rings, Bull. Amer. Math. Soc. 76 (1970)  no. 6, 1212--1225.

\bibitem{IKR}
Z.~Izhakian, M.~Knebusch, and L.~Rowen, Supertropical quadratic
forms~I, J. Pure Appl. Algebra 220, no. 1  (2016), 61--93.


\bibitem{JMR}  J.~Jun, K.~Mincheva, and L.~Rowen, Homology of module
systems, arXiv:1809.01996, J. Pure and applied algebra, to appear.


\bibitem{JuR1}  J.~Jun and L.Rowen, \emph{Categories with negation},
in ``Categorical, Homological and Combinatorial Methods in Algebra''
(AMS Special Session in honor of S.K. Jain's 80th birthday), to
appear  in  Contemporary Mathematics (2020), arXiv 1709.0318.

\bibitem{Ka1} Y.~Katsov,  \emph{Tensor products   of
functors}, Siberian J. Math. 19 (1978), 222-229, trans. from
Sirbiskii Mathematischekii Zhurnal 19 no. 2 (1978), 318--327.


\bibitem{Ka2} Y.~Katsov,  \emph{Tensor products and injective envelopes of
semimodules over additively regular semirings}, Algebra Colloquium 4
no. 2, (1997),  121--131.
\bibitem{Lor1} O.~Lorscheid,  The geometry of blueprints: Part
I: Algebraic background and scheme theory Advances in Mathematics
 229, no. 3, (2012),  1804--1846.


 \bibitem{New}  M.~H.~A.~Newman, On theories with a combinatorial definition of "equivalence,'' Ann. of Math. (2) 43 (1942), 223--243.

\bibitem{Row16} L.H.~Rowen, \textit{Algebras with a negation map}, 75 pages,
arXiv:1602.00353 [math.RA].

\bibitem{Row17} L.H.~Rowen, An informal overview of triples and systems, arXiv 1709.03174
(2017).

\bibitem{Tak} M.~Takahashi, On the bordism categories III, Math. Seminar
 Notes Kobe University 10 (1982), 211--236.

\bibitem{W}  Wikipedia, Exponential of nilpotent derivation with divided Leibniz condition
powers is endomorphism.
\end{thebibliography}
\end{document}